\documentclass[12pt]{amsart}
\usepackage[T1]{fontenc}
\usepackage[margin=1in]{geometry}
\usepackage{booktabs,bookmark,graphicx,mathtools,tikz, xcolor, siunitx}
\usepackage{lineno}
\usepackage{hyperref} 
\usepackage{amssymb}
\usepackage{ytableau}

\DeclarePairedDelimiterX\set[1]{\lbrace}{\rbrace}{#1}
\DeclarePairedDelimiterX\abs[1]{\lvert}{\rvert}{#1}

\newcommand{\QQ}{\mathbb{Q}} 
\newcommand{\ZZ}{\mathbb{Z}} 

\newtheorem{theorem}{Theorem}[section]
\newtheorem{lemma}[theorem]{Lemma}

\newtheorem{corollary}[theorem]{Corollary}

\theoremstyle{definition}

\newtheorem{example}[theorem]{Example}

\theoremstyle{remark}
\newtheorem{remark}[theorem]{Remark}

\begin{document}

\title{On the Number of 2-Hooks and 3-Hooks of Integer Partitions}


\author[Mcspirit]{Eleanor Mcspirit}
\address{Department of Mathematics, University of Virginia, Charlottesville, VA 22904}
\email{egm3zq@virginia.edu}

\author[Scheckelhoff]{Kristen Scheckelhoff}
\address{Department of Mathematics, University of Virginia, Charlottesville, VA 22904}
\email{qpz4ex@virginia.edu}


\thanks{E.M. acknowledges the support of a UVA Dean's Doctoral Fellowship.}

\maketitle

\begin{abstract}
  Let $p_t(a,b;n)$ denote the number of partitions of $n$ such that the number of $t$ hooks is congruent to $a \bmod{b}$. For  $t\in \{2, 3\}$, arithmetic progressions $r_1 \bmod{m_1}$ and $r_2 \bmod{m_2}$ on which $p_t(r_1,m_1; m_2 n + r_2)$ vanishes were established in recent work by Bringmann, Craig, Males, and Ono using the theory of modular forms. Here we offer a direct combinatorial proof of this result using abaci and the theory of $t$-cores and $t$-quotients. 
\end{abstract}


\section{Introduction and Statement of Results}

A partition $\lambda$ of an integer $n$, which is a non-increasing sequence of positive integers summing to $n$, can be represented visually by a collection of boxes arranged in left-justified rows.  The row lengths are arranged in non-increasing order and correspond to the size of each part of $\lambda$. Such a presentation is called the \emph{Ferrers-Young diagram} for $\lambda$. In such a diagram, we may refer to the cells by their place in this array; we will denote the cell in row $i$ and column $j$ by $(i,j)$. 
    
We define the \emph{hook} of $(i,j)$ to be the collection of cells $(a,b)$ such that $a=i$ and $b \geq j$ or $a \geq i$ and $b=j$. The \emph{length} of such a hook is the cardinality of this set, which we will denote $h(i,j)$. We will call a hook a \emph{$t$-hook} if its length is divisible by $t$. For example, the hook-lengths $h(i,j)$ for the partition $3+2+1$ are labeled in their respective cells below: 
$$\begin{ytableau}
  5 & 3 & 1\cr
  3 & 1\cr
  1 \cr
\end{ytableau}$$

A study of hook numbers enriches many areas where partitions naturally arise. For instance, recall that the partitions of $n$ parameterize the irreducible complex representations of the symmetric group on $n$ letters. Moreover, the degree of such a representation $p_\lambda\colon S_n \to \text{GL}(V_\lambda)$ is given by the Frame-Thrall-Robinson hook-length formula: if $\lambda$ is a partition of $n$ and $p_\lambda$ is an irreducible complex representation of $S_n$ as above, then 
        \[
        \dim(V_\lambda)= \frac{n!}{\prod_{h(i,j) \in \mathcal{H}(\lambda)}h(i,j)},
        \]
        where $\mathcal{H}(\lambda)$ is the multiset of hook-lengths of cells in the Ferrers-Young diagram for $\lambda$ \cite{FrameRobinsonThrall}.
    
    Further,
    recall the famous $q$-series identities of Euler and Jacobi:
\begin{align*}
    q \prod_{n=1}^\infty (1-q^{24n}) &= q-q^{25}-q^{49}+q^{121}+q^{169}-\dots \\
    q \prod_{n=1}^\infty (1-q^{8n})^3 &= q-3q^9+5q^{25}-7q^{49}+11q^{121}-\dots
\end{align*}

These identities appear as specializations of the Nekrasov-Okounkov hook-length formula, which is derived by taking a $z$-deformation of the generating function for the partition function. In particular, 
for any complex number $z$, we have 
\[\prod_{n=1}^\infty(1-q^n)^{z-1} = \sum_{\lambda} q^{|\lambda|} \prod_{h(i,j) \in \mathcal{H}(\lambda)} \left(1-\frac{z}{h(i,j)^2}\right).\] This result, in which hook-lengths prominently appear, constitutes a significant generalization of many new and classical number-theoretic results on $q$-series. 

We are motivated to explore the the number of $t$-hooks across all partitions of size $n$ in an attempt to obtain an analogue of Dirichlet's theorem on primes in arithmetic progressions. There has been recent work on the distribution of partitions where the number of $t$-hooks lie in a fixed progression, and it turns out that -- unlike in Dirichlet's theorem -- equidistribution does not in general hold. Going beyond unequal distribution, there are situations where such counts are identically zero. For example, consider the following results of \cite{BCMO} for congruence classes of the number of 2-hooks mod 3 for selected integers $n$:


\quad

\begin{center}
\begin{tabular}{|c|c|c|c|}
    \hline
    $n$ & $p_2(0,3;n)$ & $p_2(1,3;n)$ & $p_2(2,3;n)$  \\
    \hline
    300 & $\approx 0.7347$ & $\approx 0.2653$ & 0 \\
    \vdots & \vdots & \vdots & \vdots \\
    600 & $\approx 0.6977$ & $\approx 0.3022$ & 0 \\
    900 & $\approx 0.6837$ & $\approx 0.3163$ & 0 \\
    \vdots & \vdots & \vdots & \vdots \\
    4500 & $\approx 0.6669$ & $\approx 0.3330$ & 0 \\
    4800 & $\approx 0.6669$ & $\approx 0.3330$ & 0 \\
    5100 & $\approx 0.6668$ & $\approx 0.3331$ & 0 \\
    \hline
\end{tabular}
\end{center}

\quad

This example highlights one of the more striking examples of the unequal distribution of $t$-hooks that has been observed in recent work: see \cite{CraigPun} and \cite{BCMO}. In particular, these results construct arithmetic progressions where these counts are identically zero. The existing proofs make use of the theory of theta functions to find exact formulas for the asymptotic behavior of the number of $t$-hooks in a particular arithmetic progression. In this note, however, we argue the following result directly by appealing to the combinatorial study of hook-lengths through abaci: 

    \begin{theorem}{\text {\rm (\cite{BCMO} Theorem 1.3)}} \label{thm:main}
    Let $\left(\frac{\cdot}{\ell}\right)$ denote the Legendre symbol and $\mathrm{ord}_\ell(n)$ the $\ell$-adic valuation of $n$.
Then the following are true.
    \begin{enumerate}
        \item If $\ell$ is an odd prime and $a_1,a_2$ are  integers for which $\left(\frac{-16a_1+8a_2+1}{\ell}\right)=-1$, then we have
        \[
        p_2\left(a_1,\ell;\ell k+a_2\right)=0.
        \]
        \item If $\ell \equiv 2 \bmod{3}$ is prime and $a_1,a_2$ are  integers for which $\mathrm{ord}_{\ell}(-9a_1 + 3a_2+1)=1$, then we have 
        \[
        p_3\left(a_1,\ell^2;\ell^2 k+a_2\right)=0.
        \]
    \end{enumerate}
\end{theorem}

\begin{example}
    Let $\ell = 5$, $a_1=1$ and $a_2=1$. Observe that $\left( \frac{-16a_1+8a_2+1}{5}\right) = -1$ and
    $\mathrm{ord}_{5}(-9a_1+3a_2+1)=1$. Thus, if $n \equiv 1 \bmod{5}$, we have $p_2(1,5,n)=0$; i.e., there does not exist a partition $\lambda$ of $n$ where the number of $2-$hooks of $\lambda$ is equivalent to 1 mod 5. Likewise, if $n \equiv 1 \bmod{25},$ then $p_3(1,25;n)=0$. 
\end{example}

In the following sections we make use of a well-known bijection between integer partitions and their decompositions into $t$-cores and $t$-quotients.  We are then able to approach the theorem directly using the theory of abaci and a naturally-appearing positive definite binary quadratic form intimately related to the problem.

\section{Nuts \& Bolts}

\subsection{Cores and Quotients}

Fix an integer $t$. A partition which does not contain any $t$-hooks is called a \emph{$t$-core}. We may construct a $t$-core $\tilde{\lambda}$ from an arbitrary partition $\lambda$ by removing rim $t$-hooks from the Ferrers-Young diagram of $\lambda$. Moreover, this $t$-core is unique. 
To this end, \cite{GarvanKimStanton} describe an algorithm for finding both the core and quotient of a partition by systematically removing rim-hooks from $\lambda$ to build a $t$-quotient and arrive at the $t$-core of a partition.

The following is a well-known bijection which identifies a partition with its $t$-core and a $t$-tuple of partitions called its $t$-quotient.  In particular, let $P$ denote the space of all partitions, $c_t(P)$ all $t$-core partitions, and $q_t(P)$ the space of $t$-quotients, which is isomorphic to the direct product of $t$ copies of $P$. Then the bijection $\varphi\colon P \to c_t(P) \times q_t(P)$ is given by 
\[
\varphi(\lambda) = (\tilde{\lambda}, \lambda_0, ..., \lambda_{t-1}).
\]
In particular, we have
\begin{equation}
|\lambda|=|\tilde{\lambda}|+ t \cdot \sum_{i=0}^{t-1}|\lambda_i|.
\end{equation}

This decomposition is classical and gives rise to the fact that for a given partition $\lambda$, the size of its corresponding $t$-quotient given by $\sum_{i=0}^{t-1}|\lambda_i|$ is a count of the number of $t$-hooks contained in $\lambda$ (for example, see \cite{olsson1993combinatorics}, Theorem 3.3).  In order to make full use of the content of this bijection, we will revisit this result in the context of abaci theory. 
 
\subsection{Abaci}

We may encode the data of a partition $\lambda$ of $n$ in an abacus. Let $\lambda = \{\lambda_1 \geq \lambda_2 \geq \dots \lambda_s > 0\}$.  For each $1\leq i \leq s$ define the $i$th \emph{structure number} $B_i$ by  $B_i = \lambda_i - i + s$; i.e., the structure number $B_i$ is the hook number of the entry in the $i$th row and first column of the Young diagram.

Now, create an abacus with $t$ vertical runners labeled $0$ through $t-1$, each infinitely long. We will place beads on the runners in accordance with the structure numbers.  In particular, the $B_i$ are positive integers, thus there exist (by Euclidean division) unique integers $r_i$ and $c_i$ such that 
\[
B_i = t(r_i-1)+c_i, \quad 0 \leq c_i \leq t-1, \quad r_i \geq 1.
\]
We place a bead representing $B_i$ in the row and column position $(r_i,c_i)$ on the abacus.

For example, consider the partition $\lambda = (5, 3, 2, 1)$. We find $B_1=\lambda_1 -1 + 4 = 5-1+4 = 8$. Likewise, $B_2=5$, $B_3=3$, and $B_4=1$. The corresponding abacus with $t=3$ is 
$$\begin{tabular}{c|ccc}
1& $\cdot$ & $\circ$ & $\cdot$ \\
2& $\circ$ & $\cdot$ & $\circ$ \\
3& $\cdot$ & $\cdot$ & $\circ$ \\
\end{tabular}$$
Conversely, given an abacus with beads in positions $\{(r_i,c_i)\}$, we can construct a decreasing sequence $B_1 \geq ... \geq B_k$ by defining $B_i=t(r_i-1)+c_i$. Then we can find a corresponding partition by setting $\lambda_i=B_i+i-s$. 

\begin{lemma}
    Removing a $t$-hook from a partition $\lambda$ is equivalent to sliding a bead up one row on the abacus representing $\lambda$.
\end{lemma}

\begin{proof}
    Let $\{B_1,B_2,\dots,B_k\}$ be a set of structure numbers for the partition $\lambda$.  Removing a rim $t$-hook $T$ from $\lambda$ is equivalent to subtracting $t$ from some structure number $B_\ell$, resulting in the set of structure numbers for $\lambda \setminus T$, given by $\{B_1,\dots,B_{\ell -1},B_{\ell}-t,B_{\ell +1},\dots,B_k\}$ (Lemma 2.7.13, \cite{symmetricgrouptext}).
    
    By construction, the bead position on the abacus for $B_{\ell}$ is given by $(r_{\ell},c_{\ell})$, where $B_{\ell}=t(r_{\ell}-1)+c_{\ell}$.  Then
    \begin{align*}
        B_{\ell}-t &= t(r_{\ell}-1)+c_{\ell}-t\\
        &= t(r_{\ell}-2)+c_{\ell}\\
        &= t((r_{\ell}-1)-1)+c_{\ell},
    \end{align*}
    and thus removing a $t$-hook has the effect of sliding a bead up one row on the abacus.
\end{proof}

Since rows in the abacus representing the partition $\lambda$ are labeled with non-negative integers, and sliding a bead upward on the abacus fixes the column while subtracting 1 from the row number, the process of removing $t$-hooks can only be done finitely many times before arriving at an abacus representing the $t$-core of $\lambda$.

Furthermore, since removing a $t$-hook corresponds to sliding a bead upward on the abacus, an abacus representing a $t$-core partition has no gaps between beads in any column, and any non-empty column has a bead in the first row.  (Theorem 4, \cite{Ono_4-cores}.)  
Following this observation, we may restate (1) in terms of the $t$-hooks contained in  $\lambda$.

\begin{corollary} \label{corollary}
Let $h_t(\lambda)$ denote the number of $t$-hooks contained in $\lambda$, and let $\tilde{\lambda}$ denote its $t$-core. Then
$|\lambda| = |\tilde{\lambda}| + t \cdot h_t(\lambda).$
\end{corollary}

Additionally, we may thus denote the abaci of $t$-cores by $t$-tuples of non-negative integers, indicating the number of beads in each column.  However, there are multiple abaci which represent a single $t$-core partition:

\begin{lemma}[\cite{Ono_4-cores} Lemma 1]
    The two abaci $$\mathfrak{A}_1 = (a_0,a_1,\dots,a_{t-1})\ \ \ 
    {\text {\rm and}}\ \ \ \mathfrak{A}_2 = (a_{t-1}+1,a_0,a_1,\dots,a_{t-2})$$ represent the same $t$-core partition.
\end{lemma}

By repeatedly applying the above lemma, we may find a unique abacus representation for a given $t$-core containing zero beads in the first column. Thus, a tuple $(0, a_1, ..., a_{t-1})$ uniquely represents a $t$-core. 

\subsection{Structure Theorem for 2- and 3-Cores}

We now specialize to the cases where $t$ is 2 or 3. Using the theory of abaci, we are able to completely classify $2-$ and $3-$core partitions by looking at the divisors of $8n+1$ and $3n+1$ respectively.

\begin{theorem}
 Let $c_t(n)$ denote the number of $t$-core partitions of $n$.
 \begin{enumerate}
\item We have $c_2(n) = 1$ if $8n+1$ is an odd square, and 0 otherwise.
\item We have $c_3(n) = \sum_{d \mid 3n+1} \left(\frac{d}{3}\right).$ In particular, $c_3(n)\neq 0$ if and only if for all primes $p \equiv 2 \bmod{3}$,  $\text{ord}_p(3n+1)$ is even.
    \end{enumerate}
\end{theorem}
\begin{remark}
The equality in (1) is classical by considering the self-conjugate partitions arising from triangular numbers. The equality in (2) was first proven in \cite{Granville-Ono} by comparing coefficients of closely related modular forms. 
\end{remark}
\begin{proof}[Sketch of Proof]
$t=2$:
First, we show that $n$ is a triangular number if and only if $8n+1$ is an odd square. A triangular number has the form $n=\frac{k(k+1)}{2}$ for some integer $k$. Then 
\begin{align*}
    8n+1 &= 8\cdot \frac{k(k+1)}{2} + 1 = 4k(k+1)+1 = 4k^2+4k+1 = (2k+1)^2.
\end{align*}
Certainly, if $n$ is a triangular number, we have a partition $\lambda$ where the the parts are given by $\lambda_i = k-i+1$ for $1 \leq i \leq k$, where $n=k(k+1)/2$.  We can check that this partition is indeed a 2-core by simply noting every hook is symmetric, so the hook-length is of the form $2k+1$ for suitable $k$. 

Now suppose $\lambda$ is a 2-core partition of $n$. Then some abacus $\mathfrak{A}=(0,a)$ uniquely represents $\lambda$, and has the shape
$$\begin{tabular}{c|cc}
1& $\cdot$ & $\circ$ \\
$\vdots$ & $\vdots$ &$\vdots$ \\
$a$& $\cdot$ & $\circ$ \\
\end{tabular}$$
Note that $B_i = 2(a-i)+1, 1\leq i \leq a$, and that $a$ is the number of parts of the partition. Then recalling $\lambda_i = B_i+i-a$, we have $\lambda_i = 2(a-i)+1+i-a = (a-i)+1$, thus 
\begin{align*}
    n = \sum_{i=1}^a \lambda_i
    = \sum_{k=1}^a k = \frac{k(k+1)}{2}. 
\end{align*}

$t=3$:
Every 3-core partition can be uniquely represented by an abacus of the form $\mathfrak{A}=(0,a,b)$ for some non-negative integers $a$ and $b$.  Working backwards, we obtain an expression for $n$ in terms of the structure numbers determined by this abacus: 
	\begin{align*}
	n &= \sum_{i = 1}^{a+b} \lambda_i \\	
	&= \sum_{i = 1}^{a+b} B_i + \sum_{i = 1}^{a+b} i - \sum_{i = 1}^{a+b} (a+b) \\
	&=  \sum_{i = 1}^{a+b} B_i + \frac{(a+b)(a+b+1)}{2} - (a+b)^2.
	\end{align*}
Now, we need only compute the structure numbers. We may do this by considering the beads in column 1 and column 2 separately. We have 
	\begin{align*}
	\sum_{i = 1}^{a+b} B_i &= \sum_{i = 1}^{a} (3(i-1)+1) + \sum_{j = 1}^{b} (3(j-1)+2) \\
	&=3\cdot\frac{a(a+1)}{2} -2a + 3\cdot\frac{b(b+1)}{2} - b.
	\end{align*}
Combining with the above and simplifying, we ultimately arrive at 
\[n = a^2 - ab + b^2 + b. \]
Define $x := -a+2b+1$, $y := a+b+1$. Then 
\[3n +1= x^2-xy+y^2.\]

We now have an expression for $3n+1$ in terms of a positive definite binary quadratic form with discriminant $D=-3$. The ring of integers of the imaginary quadratic number field $K=\QQ(\sqrt{-3})$ is given by $\ZZ[\omega]$, where $\omega = \frac{1+\sqrt{-3}}{2}$. The norm on $\ZZ[\omega]$ simplifies to $N(\alpha + \omega \beta) = \alpha^2-\alpha \beta+\beta^2$. The desired equality follows from constructing a correspondence between ideals in $\mathcal{O}_K$ with norm $3n+1$ as in the proof of Theorem 4.1 in \cite{brunat2021crank}. 

In particular, 
every ideal of $\mathcal{O}_K$ is principal and factors uniquely into a product of finitely many (not necessarily distinct) prime ideals. Then
given an integer prime $p$ that divides $3n+1$, we have that $p$ is either congruent to 1 or 2 mod 3. 
In the latter case, $p$ is inert and the principal ideal $(p)$ in $\mathcal{O}_K$ is prime with norm $p^2$. Thus, such $p$ we must have an even exponent in the prime factorization of $3n+1$ in $\ZZ$. 
Then to determine whether $n$ admits a 3-core partition, we may compute the prime factorization of $3n+1$ in $\ZZ$ and check for even exponents on all primes $p$ where $p \equiv 2 \bmod 3$.
\end{proof}

\section{Proof of Theorem \ref{thm:main}}

Suppose $\ell$ is an odd prime. Write $n = \ell m +a_2$ and suppose $\lambda \vdash n$ such that $h_2(\lambda) = \ell k +a_1$, or $h_2(\lambda) \equiv a_1 \bmod{\ell}$. Denote $|\tilde{\lambda}|$ by $\tilde{n}$. Using Corollary \ref{corollary}, we may write 
$$n = \tilde{n} + 2(\ell k + a_1) = \ell m + a_2$$ so 
$\tilde{n} = -2a_1 + a_2  + \ell(m-2k)$.
Then
$$8\tilde{n} + 1 = -16a_1 + 8a_2 + 1 + \ell(8m-16k).$$

Now, $8\tilde{n} + 1 \equiv -16a_1 + 8a_2 + 1 \bmod{\ell}$. If $\left(\frac{-16a_1+8a_2+1}{\ell}\right)=-1$, then $8 \tilde{n} + 1$ cannot be an odd square. Since $\tilde{\lambda}$ was assumed to be a 2-core, we have reached a contradiction. Thus, no such $\lambda$ can exist. 
    
Next, suppose $\ell$ is a prime which is $2 \bmod{3}$. Write $n=\ell^2m+a_2$ and suppose $\lambda \vdash n $ such that $h_3(\lambda)= \ell^2 k +a_1$, or $h_2(\lambda) \equiv a_1 \bmod{\ell^2}.$
    We may write 
$$n = \tilde{n} + 3(\ell^2k+ a_1) = \ell^2m + a_2$$
 
so $\tilde{n} = - 3a_1 +a_2 + \ell^2(m-3k)$. Then
$$3\tilde{n}+1 = -9a_1 + 3a_2 + 1 + \ell^2(3m-9k).$$
Now if $\text{ord}_\ell(-9a_1 + 3a_2 + 1)=1$, then $\ell \mid -9a_1 + 3a_2 + 1$ but $\ell^2  \nmid -9a_1 + 3a_2 + 1$. Since $\ell^2 \mid \ell^2(3n-9k)$, we conclude that $\ell$ divides $3\tilde{N}+1$ but $\ell^2$ does not. Since $\tilde{\lambda}$ was assumed to be a 3-core, we have reached a contradiction; therefore no such $\lambda$ can exist. 
\hfill $\square$




\clearpage
\bibliographystyle{amsalpha}
\bibliography{3hooks}

\providecommand{\bysame}{\leavevmode\hbox to3em{\hrulefill}\thinspace}
\providecommand{\MR}{\relax\ifhmode\unskip\space\fi MR }
\providecommand{\MRhref}[2]{%
  \href{http://www.ams.org/mathscinet-getitem?mr=#1}{#2}
}
\providecommand{\href}[2]{#2}
\begin{thebibliography}{BCMO21}

\bibitem[BCMO21]{BCMO}
K.~Bringmann, W.~Craig, J.~Males, and K.~Ono, \emph{Distributions on partitions
  arising from hilbert schemes and hook lengths}, Forum of Mathematics, Sigma
  (arXiv:2109.10394), recommended for publication (2021).

\bibitem[BN22]{brunat2021crank}
O.~Brunat and R.~Nath, \emph{A crank-based approach to the theory of 3-core
  partitions}, Proceedings of the American Mathematical Society \textbf{150}
  (2022), no.~01, 15--29.

\bibitem[CP21]{CraigPun}
W.~Craig and A.~Pun, \emph{Distribution properties for t-hooks in partitions},
  Annals of Combinatorics \textbf{25} (2021), no.~3, 677--695.

\bibitem[FRT54]{FrameRobinsonThrall}
J.~S. Frame, G.~de~B. Robinson, and R.~M. Thrall, \emph{The hook graphs of the
  symmetric groups}, Canad. J. Math. \textbf{6} (1954), 316--324. \MR{62127}

\bibitem[GKS90]{GarvanKimStanton}
F.~Garvan, D.~Kim, and D.~Stanton, \emph{Cranks and {$t$}-cores}, Invent. Math.
  \textbf{101} (1990), no.~1, 1--17. \MR{1055707}

\bibitem[GO96]{Granville-Ono}
A.~Granville and K.~Ono, \emph{Defect zero {$p$}-blocks for finite simple
  groups}, Trans. Amer. Math. Soc. \textbf{348} (1996), no.~1, 331--347.
  \MR{1321575}

\bibitem[JK81]{symmetricgrouptext}
G.~James and A.~Kerber, \emph{The representation theory of the symmetric
  group}, Encyclopedia of Mathematics and its Applications, vol.~16,
  Addison-Wesley Publishing Co., Reading, Mass., 1981, With a foreword by P. M.
  Cohn, With an introduction by Gilbert de B. Robinson. \MR{644144}

\bibitem[Ols93]{olsson1993combinatorics}
J{\o}rn Olsson, \emph{Combinatorics and representations of finite groups},
  no.~20, Fachbereich Mathematik, Universit{\"a}t Essen, 1993.

\bibitem[OS97]{Ono_4-cores}
K.~Ono and L.~Sze, \emph{{$4$}-core partitions and class numbers}, Acta Arith.
  \textbf{80} (1997), no.~3, 249--272. \MR{1451412}

\end{thebibliography}

\end{document}